\documentclass[12pt,reqno]{amsart}
\usepackage{amsthm,amsfonts,amsmath,amssymb}
\usepackage[pagebackref=true,backref=true,bookmarksopen=true]{hyperref}

\newtheorem{theorem}{Theorem}
\newtheorem{Conjecture}{Conjecture}
\newtheorem{prop}{Proposition}
\newtheorem{lemma}{Lemma}
\newtheorem{remark}{Remark}

\newcommand{\R}{{\mathbb R}}
\newcommand{\RV}{{\mathcal R}}
\newcommand{\T}{{\mathbf T}}
\newcommand{\I}{{\mathbf I}}
\newcommand{\RO}{{\RV\otimes\RV}}

\newcommand{\M}{{\mathbf M}}

\begin{document}
\title{A short note on Meyers' theorem}
\author{M.A. Perelmuter}
\address{SCAD Soft Ltd., 3a Osvity Street, Kyiv, 03037, Ukraine}
\email{mikeperelmuter@gmail.com}
\keywords{Riesz transform, second order elliptic operators,  Meyers' theorem}
\subjclass[2020]{35B45 42B37 35B65}

\begin{abstract}
We consider weak solutions of the second-order elliptic equation $\operatorname{div} (\mathbf{M}\nabla u) = \operatorname{div}\mathbf{F}$ in $\Omega\subset {\mathbb R}^d$ with Dirichlet boundary conditions, where $\mathbf{M}$ is a uniformly elliptic real-valued symmetric matrix $\mathbf{M}:\Omega \rightarrow \mathbb{R}^{d\times d}$ such that $\frac{1}{K} |\xi|^2 \leq\langle \mathbf{M}\xi,\xi \rangle \leq K |\xi|^2, \forall \xi \in \mathbb{R}^d, K > 1$. We prove that $\|\nabla u\|_p \leq C\|\mathbf{F}\|_p$ for any $p\in \left(\frac{26K-16}{13K-3}, \frac{26K-16}{13K-13}\right)$.
\end{abstract}
\maketitle
\section{Introduction}
Let $\Omega \subset {\mathbb R}^d, d\geq 2$ be an open bounded domain with $C^2$ boundary. Let $\M:\Omega \rightarrow \R^{d\times d}$ be a measurable real-valued symmetric matrix that satisfies the ellipticity conditions
\begin{equation*}
\frac{1}{K} |\xi|^2 \leq\langle \M\xi,\xi \rangle \leq K |\xi|^2, \qquad  \forall \xi \in \mathbb{R}^d, K > 1.
\end{equation*}
In this note, we consider the $L^p$ regularity of solutions to the following elliptic equation in divergence form with discontinuous coefficients
\begin{equation}\label{equation}
\operatorname{div}(\M\nabla u) = \operatorname{div}\mathbf{F}  \text{ in } \Omega;\quad u = 0  \text{ on } \partial \Omega
\end{equation}
The problem of first-order regularity, namely $W^{1,p}$-estimates for weak solutions,
was solved by N.~G.~Meyers~\cite{M}. Meyers considered equation \eqref{equation} on a bounded domain with $C^1$ boundary. He showed that there exists dimension-dependent exponent $p_K > 2$  such that for all $p\in (q_K,p_K)$ (here $q_K=\frac{p_K}{p_K-1}$ is the H\"older conjugate exponent) every weak solution $u$ satisfies the higher integrability estimate
\begin{equation}\label{Meyers1}
  	\| \nabla u\|_p \leq C\|\mathbf{F}\|_p.
\end{equation}
The interval $(q_K,p_K)$ is called the \textit{Meyers interval}. Determining the optimal Meyers interval for $d>2$ remains an open problem.
\begin{Conjecture}\label{Conj}
The Meyers interval is $(\frac{2K}{K+1},\frac{2K}{K-1})$.
\end{Conjecture}
For $d=2$, Conjecture \ref{Conj} is proved by Astala-Iwaniec-Saksman \cite{AIS} (using Astala's area distortion theorem) and this result is optimal (see the example in \cite{M}).
\par
The dimension-free interval
\begin{equation}\label{ISpK}
\left(\frac{14K-12}{7K-5},\frac{14K-12}{7K-7}\right)
\end{equation}
was obtained by T.~Iwaniec and C.~Sbordone~\cite[Theorem 2]{IS2001} and currently it is the best known result.
\par
For simplicity, below we will consider the case $\Omega=\R^d$, but everything stated below easily extends to the case of domains with an appropriately smooth boundary.
\par
We will use the following notation for the \textit{Riesz transform}
$$(R_j f)(x) = p.v. \frac{\Gamma\left(\frac{d+1}{2}\right)}{\pi^{\frac{d+1}{2}}} \int\limits_{\R^d}\frac{x_j-y_j}{|x-y|^{d+1}}f(y)dy, \quad j=1,\ldots,d$$
and the \textit{vector Riesz transform} $\RV= \left(R_1,R_2,\cdots,R_d\right).$
\par
In the sequel, we will also use the notation $\|\cdot\|_{p,p}:=\|\cdot\|_{L^p(\R^d,\R^d)\to L^p(\R^d, \R^d)}$.
\par
Following \cite{IS2001}, we define the operator
$$\T:=\I+2(\RO). $$
$\T:L^2(\R^d,\R^d)\to L^2(\R^d,\R^d)$ is an isometry.
Iwaniec and Sbordone \cite[Section 11]{IS2001} showed that if
\begin{equation}\label{KBound}
\|\mathbf T\|_{p,p} < \frac{K+1}{K-1},
\end{equation}
then $p$ belongs to the Meyers interval and posed
\begin{Conjecture}\label{Conj91}\cite[Conjecture 9.1]{IS2001}
For every dimension $d \geq 2$ and every $1<p<\infty$, we have
$$\|\T\|_{p,p}=p^*-1. $$
\end{Conjecture}
\begin{remark}
Iwaniec and Sbordone used the estimate $\|\T\|_{p,p}\leq 6(p^*-1)$ from  \cite[Theorem 1]{BL}.
Combining this estimate with the Riesz--Thorin theorem yields $\|\T\|_{p,p}\leq 7p-13$ (see \cite[Lemma 9.1]{IS2001}).
As an immediate consequence, one obtains \eqref{ISpK}.
\end{remark}
The aim of this note is to improve \eqref{ISpK}.
\section{Main result}
Our main result is
\begin{theorem}\label{main} The Meyers interval contains
$$\left(\frac{26K-16}{13K-3}, \frac{26K-16}{13K-13}\right).$$
\end{theorem}
\begin{prop}\label{eq:Cassese} \cite[Proposition 4.1]{Cassese}
For $p\in(1,\infty)$ and $d\in {\mathbb N}\setminus\{0,1\}$,
\begin{equation}\label{RieszMatrix}
\|\RO\|_{L^p(\R^d,M_d(\R))\to L^p(\R^d, M_d(\R))}\le p^*-1,
\end{equation}
where $M_d(\R)$ is the space of $d\times d$ matrices over $\R$ and the matrix operator $\RO$ acts on a matrix-valued function $F\colon\R^d\to M_d(\R)$ by the rule
\begin{equation*}
((\RO)(F))_{i,l}=\sum\limits_j R_i R_j F_{j,l}.
\end{equation*}
\end{prop}
\begin{prop}\label{eq:Cassese2}
For $p\in(1,\infty)$ and $d\in {\mathbb N}\setminus\{0,1\}$,
\begin{equation}\label{RieszVector}
\|\RO\|_{p,p}\le p^*-1,
\end{equation}
where the operator $\RO$ acts on a vector-valued function $F\colon\R^d\to \R^d$ by the rule
\begin{equation*}
((\RO)(F))_{i}=\sum\limits_j R_i R_j F_{j}.
\end{equation*}
\end{prop}
\begin{proof}
To prove \eqref{RieszVector}, it suffices to consider the restriction of the operator in \eqref{RieszMatrix} to matrices with only one non-zero column and to use the fact that the subspace
$\Bigl\{F=(F_{ij}) \in L^p(\R^d,M_d(\R)):\;F_{ij}=0 \text{ for } j\neq 1\Bigr\}$
is an invariant subspace for the operator $\RO$.
\end{proof}
Now the triangle inequality gives
\begin{equation*}
\|\T\|_{p,p}\leq 2p^*-1.
\end{equation*}
The drawback of this estimate is that it does not capture the correct behavior for $p$ near $2$ ($\|\T\|_{2,2}=1$ due to the isometry of $\T$). But in the spirit of Lemma 9.1 from \cite{IS2001} we have the following
\begin{lemma}\label{InterpolLemma}
Let $1<p< \infty$, then
\begin{equation}\label{Interpol26}
\|\T\|_{p,p}\le 2.6(p^*-2)+1.
\end{equation}
\end{lemma}
\begin{proof}\footnote{The proof of Lemma 4.1 in \cite{IS2001} was left to the reader, so we give here full details for our lemma.}
If $p>6$, then $2p-1\le 2.6p-4.2$, which gives \eqref{Interpol26} for these values of $p$.
\par
Now let $2\le p\le 6$. By the Riesz--Thorin interpolation theorem applied between $L^2$ and $L^6$, $$\|\T\|_{p,p}\le \|\T\|_{2,2}^{1-\theta}\,\|\T\|_{6,6}^{\theta}\le 11^{\theta},
\qquad \frac1p=\frac{1-\theta}{2}+\frac{\theta}{6},$$
so that $\theta=\frac{3(p-2)}{2p}\in[0,1], p=\frac{6}{3-2\theta}.$
Substituting $p=6/(3-2\theta)$ into the desired inequality \eqref{Interpol26} and multiplying by $3-2\theta>0$, we find that the claim becomes $(3-2\theta)\,11^{\theta}\le 3+8.4\theta, \theta\in[0,1].$
Define
$$G(\theta)=3+8.4\theta-(3-2\theta)11^{\theta}, \qquad \theta\in[0,1].$$
We must show $G(\theta)\ge 0$ on $[0,1]$.
Differentiating, we obtain
$$G'(\theta)=8.4+11^{\theta}\bigl(2-(3-2\theta)\log 11\bigr),$$
$$G''(\theta)=11^{\theta}\log 11\,\bigl(4-(3-2\theta)\log 11\bigr).$$
$G''$ is strictly increasing in $\theta$, so $G''$ has a unique zero
$\theta_0=\frac{3}{2}-\frac{2}{\log 11}\approx 0.666$,
with $G''<0$ on $[0,\theta_0)$ and $G''>0$ on $(\theta_0,1]$. Thus
$G$ is \emph{concave} on $[0,\theta_0]$ and \emph{convex} on $[\theta_0,1]$.
\par
Since $G$ is concave on $[0,\theta_0]$, its graph lies above the chord joining
$(0,G(0))$ and $(\theta_0,G(\theta_0))$. We have $G(0)=0$, and $G(\theta_0)\approx 0.358> 0$.
Since the chord joining $(0,0)$ and $(\theta_0,G(\theta_0))$ is nonnegative on $[0,\theta_0]$, concavity gives $G(\theta)\ge 0$ for all $\theta\in[0,\theta_0]$.
\par
Since $G$ is convex on $[\theta_0,1]$, it lies above each of its tangent lines on this interval. Take $\theta_1=0.86\in [\theta_0,1]$. Let $L(\theta)=G(\theta_1)+G'(\theta_1)(\theta-\theta_1)$ be the tangent line at $\theta_1$.
Since $L$ is linear, it suffices to check its sign at the endpoints of $[\theta_0,1]$. Numerical evaluation gives
$$G(\theta_1)\approx 0.1591,\: G'(\theta_1)\approx -0.0081,\: L(\theta_0)\approx 0.1606>0, \: L(1)\approx 0.1578>0.$$
Hence $L\ge 0$ on $[\theta_0,1]$, and by convexity $G(\theta)\ge L(\theta)\ge 0$ for all $\theta\in[\theta_0,1]$.
\par
Since $\T$ is self-adjoint on $L^2$, duality yields $\|\T\|_{p,p}=\|\T\|_{p',p'}$. This allows us to extend \eqref{Interpol26} to $1<p\le2$.
\end{proof}
\begin{remark}
The calculations in the proof of this lemma can be slightly refined, and the constant $2.6$ can be replaced by $2.55$. However, this improvement is not significant.
\end{remark}
Combining Lemma \ref{InterpolLemma} with estimate \eqref{KBound}, we obtain Theorem \ref{main}.

\end{document}